\documentclass[11pt,a4paper]{amsart}
\usepackage{amscd,amssymb,amsopn,amsmath,amsthm,graphics,amsfonts,enumerate,verbatim,calc}
\usepackage[dvips]{graphicx}
\usepackage{amssymb,amsmath,amsthm,amsfonts,amsopn,url}
\usepackage{amscd,amssymb,amsopn,amsmath,amsthm,graphics,amsfonts,enumerate,verbatim,calc}
\usepackage[all]{xy}
\input xy
\xyoption{all}

\usepackage{amssymb,amsmath}

\headsep=1cm \numberwithin{equation}{section}
\newtheorem{theorem}{Theorem}[section]
\newtheorem{lemma}[theorem]{Lemma}

\theoremstyle{definition}
\newtheorem{definition}[theorem]{Definition}
\theoremstyle{remark}

\newtheorem{acknowledgement}{Acknowledgement}
\newtheorem{duplicate}{Theorem}

\newcommand{\im}{\operatorname{im}}

\newcommand{\id}{\operatorname{id}}
\newcommand{\Ext}{\operatorname{Ext}}

\newcommand{\Hom}{\operatorname{Hom}}

\newcommand{\cl}{\operatorname{cl}}

\newcommand{\symt}{\text{Sym}^2}
\newcommand{\fm}{\frak{m}}

\newcommand{\RR}{{\mathbb R}}

\begin{document}
\author{Rajsekhar Bhattacharyya}

\title
{Existence of almost Cohen-Macaulay algebras implies the existence of big Cohen-Macaulay algebras}

\address{Dinabandhu Andrews College, Kolkata 700084, India}
\email{rbhattacharyya@gmail.com}
\subjclass[2010]{13C14}

\keywords{Almost Cohen-Macaulay, Big Cohen-Macaulay, Closure operation.}

\begin{abstract}
In \cite{AB}, the dagger closure is extended over finitely generated modules over Noetherian local domain $(R,\fm)$ and it is proved to be a Dietz closure. In this short note we show that it also satisfies the `Algebra axiom' of \cite{R.G} and this leads to the following result of this paper: For a complete Noetherian local domain, if it is contained in an almost Cohen-Macaulay domain, then there exists a balanced big Cohen-Macaulay algebra over it. 
\end{abstract}

\maketitle

\section{introduction}

A big Cohen-Macaulay algebra over a local ring $(R,\fm)$ is an $R$-algebra $B$ such that some system of parameters of $R$ is a regular sequence on $B$. It is balanced if every system of parameters of $R$ is a regular sequence on $B$. Big Cohen-Macaulay algebras exist in equal characteristic \cite{Ho1}, \cite{HH5} and also in mixed characteristic when $\dim R\leq 3$ \cite{Ho3}. The existence of big Cohen-Macaulay algebras is extremely important as for example, it gives new proofs for many of the homological conjectures, such as the direct summand conjecture, monomial conjecture, and vanishing conjecture for maps of Tor.

In equal characteristic, the tight closure operation has been used to present proofs of the existence of balanced big Cohen-Macaulay modules and algebras. In \cite{D}, a list of seven axioms for a closure operation is defined for finitely generated modules over a complete local domain $R$. Any closure operation which satisfies these seven axioms is called Dietz closure. A Dietz closure is powerful enough to produce big Cohen-Macaulay modules. It is to be noted that, in a recent work of Dietz \cite{D1}, axioms are extended beyond finitely generated modules. 

In this context, it is worth it to recall the definition of an almost Cohen-Macaulay algebra (\cite{RSS}): Let $A$ be an algebra over a Noetherian local domain $(R,\fm)$, equipped  with a normalized value map $v : A\to \RR\cup\{\infty\}$. $A$ is called almost Cohen-Macaulay if each element of $((x_{1},\ldots, x_{i-1})A:_{A}x_{i})/(x_{1},\ldots,x_{i-1})A$ is annihilated by elements of sufficiently small order with respect to $v$ for all system of parameters $x_1,\ldots, x_d$ of $A$.

In \cite{AB}, using the algebra with normalized value map, we extend the definition of dagger closure to the finitely generated modules over local domain and we showed that dagger closure is a Dietz closure (\cite{AB}, Proposition 4.5) when $A$ is almost Cohen-Macaulay. As a consequence we showed that for a complete Noetherian local domain, if it is contained in an almost Cohen-Macaulay domain, then there exists a balanced big Cohen-Macaulay module over it (\cite{AB}, Corollary 4.6). 

Recently, in \cite{R.G}, an additional axiom is introduced, which is known as the `Algebra Axiom', and it is shown there that many closure operations satisfy this axiom. It is also proved that a local domain $R$ has a Dietz closure that satisfies the Algebra Axiom if and only if $R$ has a big Cohen-Macaulay algebra. In this short note we show that dagger closure as in \cite{AB} also satisfies the `Algebra axiom' of \cite{R.G}. This leads us to the following result of this paper.

\begin{duplicate}[Theorem 3.2]
For a complete Noetherian local domain, if it is contained in an almost Cohen-Macaulay domain, then there exists a balanced big Cohen-Macaulay algebra over it.
\end{duplicate}

In \cite{Ho3}, Hochster proves the existence of weakly functorial big Cohen-Macaulay algebras from the existence of almost Cohen-Macaulay algebras when $\dim R \leq 3$. Our result provides a partial extension that is not tied to the Krull dimension of the ring. 

\section{preliminary results}

Hochster and Huneke defined dagger closure for ideals in a local domain via elements of small order. See \cite{HH}. We previously extended that notion as follows:

\begin{definition} (\cite{AB})\label{def:almost closure of an ideal}
Let $A$ be a local algebra with a normalized valuation $v:A\rightarrow \RR\cup\{\infty\}$ and $M$ be an $A$-module. Consider a submodule $N\subset M$. Given $x\in M$, we say $x\in N^{v}_M$ if for every $\epsilon> 0$, there exists $a\in A$ such that $v(a)< \epsilon$ and $ax\in N$.
\end{definition}

Now we extend the definition of dagger closure to finitely generated modules (\cite{AB}) and we call this new closure as dagger  closure to local domains contained in such an $A$ as above.

\begin{definition}\label{def}
Let $(R,m)$ be a Noetherian local domain and let $A$ be a local domain containing $R$ with a normalized valuation $v:A\rightarrow \RR\cup\{\infty\}$. For any finitely generated $R$-module $M$ and for its submodule $N$, given $x\in M$ we say that $x\in N_{M}^{\bold{v}}$ if $1\otimes x \in \im(A\otimes N\to A\otimes M)_{A\otimes M}^{v}$. We call $N_{M}^{\bold{v}}$ the dagger closure of $N$ inside $M$.
\end{definition}

We recall the definition of Dietz closure.

\begin{definition}
For a Noetherian local domain, if a closure operation satisfies all the axioms of Axiom 1.1 of \cite{D}, then it is defined as Dietz closure.
\end{definition}

We restate the results of Proposition 4.5 and Corollary 4.6 of \cite{AB} in a single theorem.

\begin{theorem}
Let $R$ be a complete Noetherian local domain and let $A$ be a local domain containing $R$, equipped with a normalized valuation $v:A\rightarrow \RR\cup\{\infty\}$. Consider the dagger closure operation as defined in Definition 2.2. If $A$ is an almost Cohen-Macaulay algebra, then dagger closure becomes a Dietz closure. Moreover in this case, there exists a balanced big Cohen-Macaulay module over $R$.
\end{theorem}

As mentioned earlier, in \cite{D1}, axioms are extended beyond finitely generated modules. These new axioms should imply that dagger closure (as defined in Definition 2.2) is a Dietz closure even without assuming that the module is finitely generated. The new work also extends the definition of phantom extensions for non finitely generated modules.

We recall the definition of phanton extension for a closure operation (see \cite{HH3} and \cite{D}). Let $R$ be a ring with a closure operation $\cl$, $M$ a finitely generated $R$-module, and $\alpha:R\rightarrow M$ an injective map with cokernel $Q$. We have a short exact sequence 

\[\begin{CD}
0 @>>> R @>{\alpha}>> M @>>> Q @>>> {0.}
\end{CD}\]
Let $P_\bullet$ be a projective resolution (equivalently, free resolution since $R$ is local) for $Q$ over $R$. Then, this yields the following commutative diagram 
\[\begin{CD}
0 @>>> R @>{\alpha}>> M @>>> Q @>>> 0 \\
@AAA @AA{\phi}A @AAA @AA{\id}A @. \\
P_2 @>>> P_1 @>{d}>> P_0 @>>> Q @>>> 0. \\
\end{CD}\]
Let $\epsilon \in \Ext_R^1(Q,R)$ be the element corresponding to this short exact sequence via the Yoneda correspondence. We say that $\alpha$ is a $\cl$-phantom extension, if for above projective resolution $P_\bullet$ of $Q$, a cocycle representing $\epsilon$ in $\Hom_R(P_1,R)$ is in $\im (\Hom_R(P_0,R)\rightarrow \Hom_R(P_1,R))_{\Hom_R(P_1,R)}^{\cl}$. When closure operation is dagger closure as defined above, we call $\cl$-phantom extension as `dag-phantom' extension.

The Algebra Axiom of \cite{R.G} states that: For any closure operation $\cl$, if $R \overset{\alpha}\to M,$ is $\cl$-phantom, then the map $R \overset{\alpha'}\to \symt(M)$, sending $1 \mapsto \alpha(1) \otimes \alpha(1)$ is also $\cl$-phantom. The author proves (Theorem 3.3) that if a local domain $R$ has a Dietz closure $\cl$ that satisfies the Algebra axiom, then $R$ has a balanced big Cohen-Macaulay algebra.

\section{main result}

In this section we prove the main theorem. We begin with the following lemma.

\begin{lemma}
Let $R$ be a complete Noetherian local domain and let $A$ be an almost Cohen-Macaulay local domain containing $R$, possessing a normalized valuation map $v:A\rightarrow \RR\cup\{\infty\}$. For a finitely generated $R$-module $M$, an injection $\alpha:R\to M$ is dag-phantom if and only if for every $\epsilon >0$, there exist an element $a\in A$ with $v(a)<\epsilon$ and an $A$-module homomorphism $\id_A\otimes_R \gamma_{\epsilon}:A\otimes_R M \to A$, such that $(\id_A\otimes_R \gamma_{\epsilon}) \circ (\id_A\otimes_R \alpha)=a(\id_A\otimes_R \id_R)$.
\end{lemma}

\begin{proof}
Since $A$ is a torsion free $R$-module (\cite{R.G}, Lemma 3.9), $A\stackrel{\id_A\otimes_R \alpha}{\rightarrow} A\otimes_R M$ is injective. Observe the short exact sequence
\[\begin{CD}
0 @>>> R @>{\alpha}>> M @>>> Q @>>> 0 \\
\end{CD},\] and let $P_\bullet$ projective resolution of $Q$. This implies the following commutative diagram where $A\otimes_R P_\bullet$ is a projective resolution (equivalently, free resolution) of $A\otimes_R Q$ over $A$.

\[\begin{CD}
0 @>>> A @>{\id_A\otimes_R \alpha}>> A\otimes_R M @>>> A\otimes_R Q @>>> 0 \\
@AAA @AA{\id_A\otimes_R \phi}A @AAA @AA{\id}A @. \\
A\otimes_R P_2 @>>> A\otimes_R P_1 @>{\id_A\otimes_R d}>> A\otimes_R P_0 @>>> A\otimes_R Q @>>> 0. \\
\end{CD}\]

Now consider the following statement: For every $\epsilon >0$, there exist an element $a\in A$ with $v(a)<\epsilon$ and an $A$-module homomorphism $\id_A\otimes_R \gamma_{\epsilon}:A\otimes_R M \to A$, such that $(\id_A\otimes_R \gamma_{\epsilon}) \circ (\id_A\otimes_R \alpha)=a(\id_A\otimes_R \id_R)$. Clearly, this is true if and only if, \[a(\id_A\otimes_R \phi) \in \im(A \otimes_R \Hom_R(P_0,R) \to A \otimes_R \Hom_R(P_1,R)),\] by Lemma 5.6 of \cite{HH3}. This is equivalent to \[a(\id_A\otimes_R \phi) \in \im(A \otimes_R B \to A \otimes_R \Hom_R(P_1,R)),\] where $B$ is the module of coboundaries in $\Hom_R(P_1,R)$. Further from definition of dagger closure, this is equivalent to $\phi \in B^{\bold{v}}_{\Hom_R(P_1,R)}$. Finally, from Lemma 5.6 of \cite{HH3} again, the last statement is equivalent to the fact that the map $R\stackrel{\alpha}{\rightarrow} M$ is dag-phantom.
\end{proof}

Now we state our main result.

\begin{theorem}
Let $R$ be a complete Noetherian local domain and let $A$ be an almost Cohen-Macaulay local domain containing $R$, possessing a normalized valuation $v:A\rightarrow \RR\cup\{\infty\}$. Then there exists a balanced big Cohen-Macaulay algebra over $R$.
\end{theorem}

\begin{proof}
Consider the dagger closure operation as defined in Definition 2.2 of the previous section and from Theorem 2.4, we know that it is a Dietz closure. Now, if we can show that this dagger closure also satisfies `Algebra axiom' as stated in \cite{R.G}, then using Theorem 3.3 of \cite{R.G}, we can conclude.

To see that dagger closure satisfies the `Algebra axiom' of \cite{R.G}, we assume that injective map $R\stackrel{\alpha}{\rightarrow} M$ is dag-phantom, where $M$ is a finitely generated $R$-module. From Lemma 3.1, this is equivalent to the fact that for every $\epsilon > 0$, there exist an element $a\in A$ with $v(a)<\epsilon$ and an $A$-module homomorphism $\id_A\otimes_R \gamma_{\epsilon}:A\otimes_R M \to A$, such that $(\id_A\otimes_R \gamma_{\epsilon}) \circ (\id_A\otimes_R \alpha)=a(\id_A\otimes_R \id_R)$. We need to show that $\alpha':R \to \symt(M)$ is also dag-phantom, where $\alpha'(1)= \alpha(1)\otimes \alpha(1)$. For $\epsilon>0$, take $a\in A$ with $v(a)< \epsilon/2$ such that $(\id_A\otimes_R \gamma_{\epsilon/2})\circ (\id_A\otimes_R \alpha)={a}(\id_A\otimes_R \id_R)$. Set $b=a^2$ and from properties of valuation, $v(b)<\epsilon$. Since symmetric algebra commutes with base change, we define $\id_A\otimes_R \gamma_{\epsilon}':A \otimes_R \symt(M) \to A$ by
\[(\id_A\otimes_R \gamma_{\epsilon}')(a \otimes (m \otimes m'))=a\gamma_{\epsilon/2}(m)\gamma_{\epsilon/2}(m').\]
 Then 
 \[(\id_A\otimes_R \gamma_{\epsilon}')((\id_A\otimes_R \alpha')(1))=(\id_A\otimes_R \gamma_{\epsilon}')(1 \otimes (\alpha(1) \otimes \alpha(1)))=(\gamma_{\epsilon/2}(\alpha(1)))^2={a^{2}}=b,\] since \[\gamma_{\epsilon/2}(\alpha(1))= 1\otimes \gamma_{\epsilon/2}(\alpha(1))= (\id_A\otimes_R \gamma_{\epsilon/2}\alpha)(1\otimes 1)= (\id_A\otimes_R \gamma_{\epsilon/2}) \circ (\id_A\otimes_R \alpha)(1\otimes 1) =\] \[a(\id_A\otimes_R \id_R)(1\otimes 1)= a.\] Again, using Lemma 3.1, we get that $\alpha'$ is also dag-phantom and thus we conclude.
\end{proof}
\begin{acknowledgement}
I would like thank the referee for careful reading of the paper and for all the valuable comments and suggestions for the improvement of the paper.
\end{acknowledgement}

\end{document}